\documentclass[journal,twoside,web]{ieeecolor}
\usepackage{generic}
\usepackage{comment}

\usepackage{empheq}

\usepackage{amsmath,amsfonts,amssymb,amscd,latexsym,enumerate,bbm}

\usepackage{amsthm,array}
\newtheorem{theorem}{Theorem}[section]
\newtheorem{proposition}[theorem]{Proposition}
\newtheorem{lemma}[theorem]{Lemma}
\newtheorem{corollary}[theorem]{Corollary}
\theoremstyle{definition}

\newtheorem{assumption}[theorem]{Assumption}
\newtheorem{remark}[theorem]{Remark}

\usepackage{arydshln}

\usepackage{cite}
\usepackage{dsfont}

\usepackage{mathtools}
\usepackage{dashbox}

\usepackage{enumerate}
\usepackage{bm}

\usepackage{tikz-network}
\usetikzlibrary{positioning}

\DeclareMathOperator{\Tr}{Tr} 
 
\DeclareMathOperator{\proj}{proj}
\DeclareMathOperator{\cl}{cl}

\DeclareMathOperator{\col}{col}

\DeclareMathOperator{\sol}{SOL}

\def\cali{{\mathcal I}}

\def\call{{\mathcal L}}

\def\caln{{\mathcal N}}

\def\cals{{\mathcal S}}
\def\calt{{\mathcal T}}
\def\calu{{\mathcal U}}

\def\calx{{\mathcal X}}

\def\caly{{\mathcal Y}}

\DeclareMathAlphabet{\mymathbb}{U}{bbold}{m}{n}

\newcommand{\bze}{\bm{0}}

\newcommand{\R}{\mathbb{R}}

\DeclareMathOperator*{\argmax}{arg\,max}
\DeclareMathOperator*{\argmin}{arg\,min}

\usepackage[prependcaption,colorinlistoftodos]{todonotes}

\usepackage[normalem]{ulem}

\usepackage{setspace}
\usepackage{url}
\makeatletter
\let\NAT@parse\undefined
\makeatother
\usepackage[hidelinks]{hyperref}
\newcommand{\oprocendsymbol}{\hbox{$\bullet$}}
\newcommand{\oprocend}{\relax\ifmmode\else\unskip\hfill\fi\oprocendsymbol}

\newcommand{\squaresymbol}{\hbox{$\square$}}
\newcommand{\squarend}{\relax\ifmmode\else\unskip\hfill\fi\squaresymbol}

\allowdisplaybreaks

\title{ \LARGE 
\bf{Dynamic interventions with limited knowledge in network games}
}

\author{Mehran~Shakarami, Ashish~Cherukuri, Nima~Monshizadeh
\thanks{Mehran Shakarami, Ashish Cherukuri, and Nima Monshizadeh are with the Engineering and Technology Institute, University of Groningen, 9747AG, The Netherlands, {\tt\small m.shakarami@rug.nl, a.k.cherukuri@rug.nl, n.monshizadeh@rug.nl.}}%
}
\begin{document}

\maketitle
\begin{abstract}
This paper studies the problem of intervention design for steering the actions of noncooperative players in quadratic network games  to the social optimum.  The players choose their actions with the aim of maximizing their individual payoff functions, while a central regulator uses interventions to modify their marginal returns and maximize the social welfare function. This work builds on the key observation that the solution to the steering problem depends on the knowledge of the regulator on the players' parameters and the underlying network. We, therefore, consider different scenarios based on limited knowledge and propose suitable static, dynamic and adaptive intervention protocols. We formally prove convergence to the social optimum under the proposed mechanisms. We demonstrate our theoretical findings on a case study of Cournot competition with differentiated goods.
\end{abstract}

\section{Introduction}

Network games have emerged as a powerful tool for studying the scenarios where the well-being of individuals depends on their own decisions as well as the actions of their neighbors in an interaction network. These games have a broad spectrum of applications such as studying crime networks \cite{ballester2006s}, pricing in social networks \cite{candogan2012optimal,bloch2013pricing}, public good provision \cite{bramoulle2007public}, firm competition \cite{acemoglu2015networks}, and telecommunication \cite{altman2006survey}. We refer to \cite{parise2019variational} for a systematic analysis of the outcome of network games via the use of variational inequalities. In economics, the problem of influencing the outcome of network games by interventions has been of great interest, and this has led to various works typically studying the effects of the network topology on  optimal policies, see, e.g., \cite{ballester2006s,galeotti2020targeting,demange2017optimal,parise2021analysis}.

Generally speaking, noncooperative games involve players who are self-interested/selfish and pursue their own well-being. Such selfish behavior of the players entails degradation of performance in comparison to the scenarios where the players would cooperate to maximize the social welfare. The deterioration in performance has led to the definitions of two performance metrics termed the price of anarchy \cite{koutsoupias1999worst} and the price of stability \cite{anshelevich2008price}, and their quantification is extensively studied in different applications such as resource allocation \cite{johari2005efficiency}, congestion games \cite{christodoulou2005price,roughgarden2002bad}, and supply chains \cite{perakis2007price}.

An active line of research concerns improving the performance of noncooperative games
and realigning the preferences of the players with the social optimum through interventions. To this end, a central regulator provides incentives to coordinate the players and alter 
their strategies towards the social optimum. The main challenge, however,  
is that optimal incentives depend on private information of the players, generally unknown to the regulator \cite{hayek1945use}. The celebrated Vickrey–Clarke–Groves (VCG) mechanism \cite{vickrey1961counterspeculation} is adopted in different disciplines, and especially  in economics, to address this problem. In this 
setup, the  mechanism generates a payment rule with the aim of incentivizing the players to announce their private information to the regulator. This information is then  used to reach to the social optimum, see \cite{roughgarden2010algorithmic} for more details on the topic.

Another methodology for enhancing the performance and achieving the social optimum in noncooperative games is to  exploit  control-theoretic tools. In this case, the players do not report their private information, but their actions are observed over time by the regulator. The problem is then regarded as a feedback control problem where the desired outcome is the social optimum and the control effort is implemented through interventions \cite{alpcan2009nash}. Devising suitable control laws is straightforward when the regulator has \emph{perfect information} on the game and the payoffs of the players, whereas it becomes much more intricate when   some of the players' private information and/or network level parameters are unknown.
To overcome this lack of information,  dynamical protocols are proposed in  \cite{alpcan2009nash,alpcan2009control,barrera2014dynamic}. In \cite{alpcan2009nash}, a dynamic pricing mechanism is devised that solves the problem for players with separable utility functions. When the utility functions are non-separable, side information is used in \cite{alpcan2009control}  for convergence to the social optimum. In particular, the pricing mechanism employs the utility functions evaluated at the Nash equilibrium. In the context of congestion control, the mechanism presented in	\cite{barrera2014dynamic} guarantees convergence assuming that the network manager knows the aggregate flow on each link as well as the delay-cost experienced by the users.
These mechanisms are not generally applicable to network games since the players' payoff functions are non-separable and the information available to the regulator is limited.

In this work, we address the problem of steering the actions of noncooperative players in quadratic \textit{network games} to the solution of the social welfare maximization problem. We consider selfish players who maximize their individual payoff functions by following pseudo-gradient dynamics. The regulator, on the other hand, is aimed at nudging the players towards the social optimum, and to do this, she modifies the marginal returns of the players through interventions. Essential to our results is the observation that the choice of interventions structurally depend on the information available to the regulator. Therefore, we 
differentiate among notable cases concerning the knowledge available to the regulator: full game information, the network structure or an estimate of social optimum. Unavailability of such information gives rise to a fourth scenario where an adaptive control scheme is proposed to achieve the social optimum. We provide analytical convergence guarantees for all the proposed protocols, and accompany our findings with a numerical case study of Cournot competition.

The structure of the paper is as follows. Notations and preliminaries are provided in Section \ref{sec:prelim}. Section \ref{sec:problem} discusses the network game model and characterizes the optimization problem faced by the regulator. Section \ref{sec:intervention} includes the intervention protocols and presents their convergence guarantees to the social optimum. The case study is provided in Section \ref{sec:sim}, and  concluding remarks and future research directions are stated in Section \ref{sec-conclusion}. Existence of a unique social optimum and  boundedness analysis of the adaptive mechanism are presented in the appendix.

\section{Notations and preliminaries}\label{sec:prelim}
This section introduces notational conventions and provides a few  basic notions on convex analysis.

\subsection{Notations}
The set of real and nonnegative real numbers are denoted by $ \R $ and $ \R_{\geq 0} $, respectively. We denote the standard Euclidean norm by $\|\cdot \|$. The symbol $\bze$ denotes a vector/matrix of all zeros. 
For given vectors $ x_1,\cdots,x_m \in\R^n$, we use the notation $\col(x_i):=\big[x_1^\top ,\cdots,x_m^\top\big]^\top$. We use $P\succ \bze \,(\prec \bze)$ to denote that $P=P^\top \in\R^{n\times n}$ is positive definite (negative definite). Given a matrix $P=P^\top \in\R^{n\times n}$, we denote its Frobenius norm by $\|P\|_\mathrm{F}=\sqrt{\Tr(P^\top P)}$ where $\Tr(\,\cdot\,)$ is the trace operator. Moreover, the notation  $\lambda_i(P)$ with $i\in\{1,\dots,n\}$ denotes the eigenvalues of $P$, and $ \lambda_{\min}(P) $ and  $ \lambda_{\max}(P) $ are the minimum and the maximum eigenvalues of $ P $, respectively. The weighted Euclidean norm of a vector $ x\in\R^n $ is given by $ \|x\|_P=\sqrt{x^\top P x} $ where $ P\succ \bze $. A function $ F:\R^n\to \R^n $ is hypomonotone, monotone, and strongly monotone, respectively,  if it satisfies $ (x-y)^\top(F(x)-F(y))\geq \mu \|x-y\|^2 $ for all $ x,y\in \R^n $, with $ \mu\leq 0 $, $ \mu=0 $, and $ \mu>0 $, respectively. For a piecewise continuous function $x:[0,\infty)\to \R^n$, we  define the $\call_\infty$ and $\call_2$ norms as $\|x\|_\infty:=\sup_{t\geq 0}\|x(t)\|$ and $\|x\|_2:=\big(\int_0^\infty \|x(\tau)\|^2 d\tau \big)^{\frac{1}{2}}$, respectively. Moreover, we say $x\in\call_\infty$ when $\|x\|_\infty$ is finite, and $x\in\call_2$ when  $\|x\|_2$ is finite.

\subsection{Convex analysis}
Consider a nonempty, closed and convex set $\calx \subseteq \R^n$. We denote the projection of a point $ z\in \R^n $ on to the set $ \calx $ by 
$ \proj_\calx (z):=\argmin\nolimits_{y\in\calx}\|y-z\| $. Given a point $ x\in \calx $, the set  $ \caln_{\calx}(x):=\big\{y\in\R^n \mid y^\top(s-x)\leq 0,\forall s\in\calx\big\} $ is the normal cone to $ \calx $ at $ x $, and the tangent cone  is denoted by $ \calt_\calx(x):=\cl\left(\cup_{y\in\calx}\cup_{\lambda>0} \lambda (y-x)\right) $ where $ \cl(\cdot) $ is the closure. Given a point $ z\in\R^n $, we denote its projection on to $ \calt_\calx(x) $ by $\Pi_{\calx}(x,z):= \proj_{\calt_\calx(x)}(z) $.  It also follows from Moreau's decomposition theorem \cite[Thm. 3.2.5]{hiriart1996convex} that $ z=\proj_{\caln_\calx(x)}(z)+\proj_{\calt_\calx(x)}(z) $. Given the set $ \calx $ and a map $ F:\calx \to \R^n $, the variational inequality problem VI$ (\calx,F)$  consists of finding a point $ \bar x\in \calx $ such that $ (x-\bar x)^\top F(\bar x)\geq 0 $ for all $ x\in \calx $. We write $ \sol(\calx,F) $ to denote the set of solutions to  VI$ (\calx,F)$.

\section{Problem formulation}\label{sec:problem}
We consider a game with the population of $ \cali:=\{1,\dots,n\} $ players/agents that interact repeatedly with a central regulator as well as with each other according to an underlying interaction network. We denote the adjacency matrix of this network by $ P\in \R^{n\times n} $ where $ P_{ij}\in [0,1] $ denotes the influence of player $ j $'s strategy/action on the utility function of player $ i $. We assume that the network has no self loop, thus $ P_{ii}=0 $ for all $ i\in\cali $, and the set of neighbors of player $ i $ is denoted by $ \caln_i=\{j\in\cali\mid P_{ij}>0\} $. 

Each  player $ i\in\cali $ is associated with a payoff function $ U_i(x_i,z_i(x),u_i) $ that depends on her own action $ x_i\in \calx_i\subseteq \R $, the aggregate of her neighbors' actions
\begin{equation}\label{z-def}
z_i(x):=\sum_{j\in\caln_i}P_{ij}x_j
\end{equation}
with $ x=\col(x_i) $, and a scalar \emph{intervention} $ u_i$ which will be determined by the central regulator. We restrict our attention to linear quadratic payoff functions of the form
\begin{equation}\label{u-pay-def}
U_i\big(x_i,z_i(x),u_i\big)=W_i\big(x_i,z_i(x)\big)+x_iu_i
\end{equation}
with 
\begin{equation}\label{w-pay-def}
W_i\big(x_i,z_i(x)\big):=-\frac{1}{2}x_i^2+x_i\big(az_i(x)+b_i\big),
\end{equation}
where $ a\in\R\setminus\{0\} $ captures the impact of neighbors aggregate actions $ z_i(x) $ and $ b_i\in\R $ is the  standalone marginal return. 
The payoff function $ W_i $ is used in the literature to model peer effects in social and economic processes,  see e.g.  \cite{ballester2006s,bramoulle2014strategic,corbo2007importance}. The term $ x_iu_i $ is included to capture the intervention of the central regulator in modifying the standalone marginal return $ b_i $ to $ b_i+u_i $ \cite{galeotti2020targeting,parise2021analysis}.

In our setup, the interventions $ u=\col(u_i)$ take values from a set $\calu \subseteq \R^n$.
The action and intervention constraint sets satisfy the following assumption.

\begin{assumption}\label{asm:u-x-set}
The constraint sets $ \calx_i\subseteq \R $ and $ \calu\subseteq \R^n $ are nonempty, closed and convex.
Moreover, the set $ \calu $ contains the origin.
\oprocend
\end{assumption}

\begin{remark}\label{rem:x-u-sets}
We note that while the constraints on the action set are local, namely $x_i\in \calx_i$, the interventions constraint set $\mathcal{U}$ allows both local, e.g. $ \calu=\R_{\geq 0}^n $, and coupled constraints, e.g. $\calu = \{u\in\R^n \mid  \|u\|\leq c \}$ for some $ c>0 $.
Another notable example is given by $ \calu=\{u\in\R^n \mid u_i\in\R,\ \forall i\in\overline \cali\text{ and } u_i=0,\ \forall i\in\cali\setminus\overline\cali\} $ which can accommodate the case where the regulator applies the intervention to a subset of players only.  
\oprocend
\end{remark}

\noindent
\textit{Problem overview.} 
The players are noncooperative and merely interested in maximizing their individual payoff functions by choosing their actions. 
This selfish behavior causes loss of efficiency with respect to the situation in which the players would cooperate to maximize the total payoff. The central regulator, on the other hand, is aimed at coordinating the players and avoiding the efficiency loss. To this end, she changes the players' standalone marginal returns through suitable interventions.  

In the next two subsections, we discuss the dynamic model capturing the strategies of the players, and characterize the optimization problem faced by the regulator.
\subsection{Players' strategy}
Each player aims at maximizing her individual payoff function given the aggregated actions of her neighbors and the current value of the intervention signal. To capture this, 
we consider that the action of each player $ i\in\cali $ evolves over time according to the following pseudo-gradient  dynamics\footnote{See \cite{de2018distributed,gadjov2018passivity,shakarami2022distributed} for further applications of continuous pseudo-gradient dynamics in the context of distributed Nash equilibrium seeking for noncooperative games.}:
\[
\dot x_i(t)= \Pi_{\calx_i}\left(x_i(t), \frac{\partial U_i}{\partial x_i} \Big(x_i(t), z_i\big(x(t)\big), u_i(t)\Big)\right),
\]
where $u_i(t)$ is the intervention designed by the regulator. 
Noting the definition of $ z_i(x)$ given by \eqref{z-def} and the fact that $ P_{ii}=0 $, we can rewrite dynamics above as
\begin{equation}\label{x-indiv}
\dot{x}_i(t)=\Pi_{\calx_i}\Big(x_i(t),-x_i(t)+a\sum_{j\in\cali}P_{ij}x_j(t)+b_i+u_i(t)\Big).
\end{equation}

Note that in the case of no intervention, i.e., $ u_i(t)\equiv 0 $, the equilibrium of \eqref{x-indiv} coincides with the Nash equilibrium of the game, namely the action profile $x^\text{NE}=\col(x_i^\text{NE})$ satisfying
\begin{equation*}
x_i^\text{NE}  \in \argmax_{y_i\in\calx_i} W_i\big(y_i,z_i(x^\text{NE})\big),\quad \forall i\in\cali,
\end{equation*}
where $ W_i $ is given by \eqref{w-pay-def}. The Nash equilibrium $ x^\text{NE} $ 
can also be expressed as a solution of the variational inequality VI($\calx,F$) where $ \calx=\prod_{i\in\cali}\calx_i $ and $ F(x):=(I-aP)x-b $.\footnote{Existence of a Nash equilibrium follows from analogous arguments to the proof of  \cite[Cor. 4.2]{basar1999dynamic}, and the relation in \eqref{e:Nash-exp} is satisfied using \cite[Prop. 1.4.2]{facchinei2007finite}.}
That is,
\begin{equation}\label{e:Nash-exp}
x^\text{NE} \in \sol(\calx,F).
\end{equation}

Next we look at the problem from the regulator's side.

\subsection{Regulator's objective}
The central regulator aims to implement suitable interventions to coordinate the players and maximize the total payoff. More precisely, she aims at designing the intervention signal $ \col(u_i(t)) $  such that the actions of the players converge to a \emph{social optimum} $x_\mathrm{opt} $,  defined as a solution of the \emph{social welfare} maximization problem:
\begin{equation}\label{social}
x_\mathrm{opt}\in \argmax_{y\in\calx}\sum_{i\in\cali} W_i\big(y_i,z_i(y)\big),
\end{equation} 
where $ y=\col(y_i) $ and $ W_i $ is given by \eqref{w-pay-def}. Any social optimum $x_\mathrm{opt} $ is also a solution to  the following variational inequality problem \cite[Prop. 2.1.2]{bertsekas99nonlinear}:
\begin{equation}\label{x-opt-gen}
x_\mathrm{opt}\in\sol(\calx,H),
\end{equation}
where $ H(x)=\big(I-a(P+P^\top)\big)x-b $. 
Note that $-H$ is the gradient of the social welfare function $ \sum_{i \in \cali} W_i\big(x_i,z_i(x)\big) $.

Observe that $x_{\mathrm{opt}}$ differs from  the Nash equilibrium in \eqref{e:Nash-exp}. The regulator, therefore, aims to designing intervention mechanisms that solve the following problem:

\medskip{}
\noindent
\textit{Problem formulation.}
Design intervention mechanisms $ u\in\calu $  that asymptotically steer the action profile $x$ of the players in \eqref{x-indiv} 
to the social optimum $x_{\mathrm{opt}}$ given by \eqref{x-opt-gen}.
\medskip{}

\section{Intervention protocols}\label{sec:intervention}

Before proceeding with the intervention protocols, we discuss existence of a unique social optimum and comment on the feasibility of the formulated problem.

\begin{lemma}\label{lem:soci-uniq}
Let Assumption \ref{asm:u-x-set} hold. Then the social welfare maximization problem \eqref{social} has a unique solution  if 
\begin{equation}\label{e:iff-sc}
\max_{i\in\cali} a\,\lambda_i(P+P^\top)<1.     
\end{equation}

\end{lemma}
\begin{proof}
See the appendix.
\end{proof}

We note that the sufficient condition \eqref{e:iff-sc} is in general necessary if one looks at arbitrary constraint set $ \calx_i $ satisfying Assumption \ref{asm:u-x-set}. 
A notable example is given by $ \calx_i=\R $ for all $i\in\cali$; see \cite[Lem. II.1]{shakarami2021adaptive}.

Motivated by Lemma \ref{lem:soci-uniq}, we impose the following standing assumption throughout the paper.

\begin{assumption}\label{asm:opt}
The adjacency matrix $ P\in\R^{n\times n} $ and the parameter $ a\in\R $ satisfy $ \max_{i\in\cali} a\,\lambda_i(P+P^\top)<1 $.
\oprocend
\end{assumption}

\begin{remark}\label{rem:uniq-xopt}
The matrix $ P+P^\top $ is symmetric with the diagonal elements equal to zero. This implies that the matrix $ P+P^\top $ has only real eigenvalues and their sum is zero. Hence,
\begin{equation*}
\lambda_{\min}(P+P^\top)<0<\lambda_{\max}(P+P^\top).
\end{equation*}
It follows from the above inequalities  that  Assumption \ref{asm:opt} is satisfied if and only if either (i) $ a>0 $ and $ a\,\lambda_{\max}(P+P^\top)<1 $ or (ii) $ a<0 $ and $ a\,\lambda_{\min}(P+P^\top)<1 $. 
\oprocend
\end{remark}

As a consequence of Assumption \ref{asm:opt}, the  social welfare function on the right-hand side of \eqref{social} is strongly concave and
thus admits a unique maximizer that is also the  solution of \eqref{x-opt-gen}, namely
\begin{equation}\label{x-opt}
x_\mathrm{opt}=\sol(\calx,H),
\end{equation}
with $ H(x)=\big(I-a(P+P^\top)\big)x-b $.

Having established the uniqueness of the social optimum $x_\mathrm{opt}$, we shift our attention to the feasibility of the problem formulated at the end of the previous section.

Noting \eqref{x-indiv}, we recall that the action profile evolves according to the following projected pseudo-gradient dynamics:
\begin{equation}\label{ne-alg}
\dot{x}(t)=\Pi_{\calx}\big(x(t),-F\big(x(t)\big)+u(t)\big),
\end{equation}
where $u\in \mathcal{U}$ and
\begin{equation}\label{e:f-map}
F(x)=(I-aP)x-b.
\end{equation}
The dynamics \eqref{ne-alg} at steady-state reads as $
\bze=\Pi_{\calx}\big(\bar x,-F( \bar x)+\bar u\big)
$
for constant action-intervention pairs $(\bar x, \bar u)\in \calx \times \calu$. We thus deduce from Moreau’s decomposition theorem that $\bze = -F( \bar x)+\bar u-\proj_{\caln_{\calx}(\bar x)}\big(-F( \bar x)+\bar u\big)$, or equivalently $-F( \bar x)+\bar u\in \caln_{\calx}(\bar x)$.
By \cite[Ex. 6.13]{rockafellar2009variational}, the pair $(\bar x, \bar u)$ satisfies the latter inclusion only if $\bar x$ belongs to the set below:
\footnote{Note that in \eqref{feas-set}, the variational inequality problem VI$\big(\calx,F-\bar u\big)$ has a  unique solution since $F$ is strongly monotone (see \eqref{F-strong-monotone} and \cite[Thm. 2.3.3]{facchinei2007finite}).} 
\begin{equation}\label{feas-set}
\cals :=\Big\{\bar x\in\calx \mid \exists \bar u\in\calu \;{\text{ such that}}\; \bar x=\sol\big(\calx,F-\bar u\big) \Big\}.
\end{equation}
The set $\mathcal{S}$ contains all assignable equilibria (action profile) of \eqref{ne-alg}, which 
necessitates the following assumption on $x_\mathrm{opt}$.
\begin{assumption}\label{asm:xopt_omega}
The social optimum $ x_\mathrm{opt}$ given by \eqref{x-opt} belongs to the set $\cals$ in \eqref{feas-set}.
\oprocend
\end{assumption}

\noindent\textit{The role of limited knowledge.} In what follows, we provide several intervention protocols that are able to steer the action profile towards the social welfare $x_{\mathrm{opt}}$. Key to our results is the observation that the suitable intervention depends on the knowledge of the regulator on the underlying game parameters.

We emphasize that Assumptions \ref{asm:u-x-set}, \ref{asm:opt}, and \ref{asm:xopt_omega} are assumed to hold throughout this section.

\subsection{Static open-loop intervention}
The first case that we consider is where the regulator has full access to the game information, i.e.,  $ (aP,b) $ and $ \calx_i $'s. The regulator, therefore, can use this knowledge to compute  $ x_\mathrm{opt} $ and its corresponding intervention $ u_\mathrm{opt}\in\calu $, with $\sol(\calx,F-u_{\mathrm{opt}}) = x_{\mathrm{opt}}$. Note that such $ u_\mathrm{opt}$ exists by  Assumption \ref{asm:xopt_omega}.
The regulator can then implement the protocol $ u(t)\equiv u_\mathrm{opt} $ to steer the action profile to $ x_\mathrm{opt} $. 
This is formalized in the following proposition.

\begin{proposition}\label{prop:open_loop}
Consider the pseudo-gradient dynamics \eqref{ne-alg}. 
Let $u_{\mathrm{opt}}\in \calu$ be such that $ \sol(\calx,F-u_{\mathrm{opt}}) = x_{\mathrm{opt}}$.
Then, for any initial condition $ x(0)\in\calx $, the  \emph{static open-loop} intervention $u(t) \equiv u_{\mathrm{opt}}$ 
steers the action profile $ x(t) $ to the social optimum $ x_\mathrm{opt} $. Moreover, $u_{\mathrm{opt}}$ satisfies
\begin{equation}\label{u-static}
u_{\mathrm{opt}} = 	(I-aP) x_\mathrm{opt}-b+v,
\end{equation}
for some $ v\in\caln_{\calx}(x_\mathrm{opt})$.
\end{proposition}
\begin{proof}
We first use the relation $x_{\mathrm{opt}} = \sol(\calx,F-u_{\mathrm{opt}})$  to show that $u_\mathrm{opt}$ admits the form \eqref{u-static}. To see this, note that 
\begin{equation*}
(y-x_\mathrm{opt})^\top\big( F(x_\mathrm{opt})-u_\mathrm{opt} \big)\geq 0,\quad \forall y\in\calx.
\end{equation*}
This implies that $v:= -F(x_\mathrm{opt})+u_\mathrm{opt}\in\caln_{\calx}(x_\mathrm{opt})  $.
The latter yields $u_\mathrm{opt}=F(x_\mathrm{opt})+v $, which together with \eqref{e:f-map} establishes \eqref{u-static}. 

Next we prove that the dynamics  \eqref{ne-alg} under the input \eqref{u-static}, has a unique solution $ x(t) $ that convergences to the social optimum. In this regard, we rewrite the overall dynamics as follows:
\begin{equation*}
\dot{x}=\Pi_{\calx}\big(x,-T(x)\big),
\end{equation*}
where $ T(x):=F(x)-F(x_\mathrm{opt})-v $. We note that for the mapping $ F $, the following holds:
\begin{equation}\label{F-strong-monotone}
\begin{split}
	(x-y)^\top(F(x)-F(y))&=\|x-y\|_{(I-\frac{1}{2}a(P+P^\top))}^2\\
	&\geq\frac{1}{2}\|x-y\|^2,\quad \forall x,y\in\R^n,
\end{split}
\end{equation}
where we have used Assumption~\ref{asm:opt} to obtain the inequality. This means that $F$ is strongly monotone, and in turn, the mapping $ T $ is also strongly monotone. In addition, the set $ \calx $ is closed and convex.
It then follows from \cite[Thm. 1]{brogliato2006equivalence} that, for any initial condition $ x(0)\in\calx $, the above dynamics has a unique solution $ x(t) $ for all $ t\geq 0 $.\footnote{A map $x:[0,\infty)\to \calx$ is a (Carath\'{e}odory) solution of the projected dynamical system $ \dot{x}=\Pi_{\calx}\big(x,-T(x)\big) $ if it is absolutely continuous and satisfies  $\dot x(t)=\Pi_\calx \big(x(t),-T(x(t))\big)$ for almost all $ t \geq 0$.}

Next consider the Lyapunov candidate $ V(x)=\frac{1}{2}\|\tilde x\|^2 $ with $ \tilde x=x-x_\mathrm{opt} $. The time-derivative of the evolution of $ V $ along the solution of the system satisfies
\begin{align*}
\nabla V(x)^\top \Pi_{\calx}\big(x,-T(x)\big)=&-\tilde x^\top T(x)\\
&-\tilde x^\top \proj_{\caln_{\calx}(x)}\big(-T(x)\big),
\end{align*}
where we used Moreau's decomposition theorem. Note that $ -\tilde x^\top \proj_{\caln_{\calx}(x)}\big(-T(x)\big)\leq 0 $ as $ x,x_\mathrm{opt}\in\calx $. It then follows from the definition of $ T$ that
\begin{align*}
\nabla V(x)^\top \Pi_{\calx}\big(x,-T(x)\big)\leq-\tilde x^\top \big(F(x)-F(x_\mathrm{opt})\big)+\tilde x^\top v.
\end{align*}
Recalling that $ v\in\caln_{\calx}(x_\mathrm{opt}) $, we have $ \tilde x^\top v\leq 0 $, and in turn, we obtain
\begin{equation*}
\nabla V(x)^\top \Pi_{\calx}\big(x,-T(x)\big)\leq -\frac{1}{2}\|\tilde x\|^2,
\end{equation*}
where we have used \eqref{F-strong-monotone}. The above inequality implies that $ V $ decreases monotonically along the solution of the closed-loop dynamics and the action profile $ x(t) $ converges to $ x_\mathrm{opt} $.
\end{proof}

\subsection{Static feedback intervention}
We next consider the case where the regulator has only access to $ aP $, but neither $b$ nor $ \calx_i $'s. This means that the regulator has complete knowledge about the network topology and the impact of the actions of the players on each other. Leveraging this information, we show that under a weak coupling condition, the regulator can steer the players to the social optimum by employing a static state feedback protocol.

\begin{proposition}\label{prop:static-fb}
Consider  the pseudo-gradient dynamics \eqref{ne-alg}. 	Assume that $ aP^\top x_\mathrm{opt}\in\calu $ and 
\begin{equation}\label{ap-weak-coupling}
\|aP\|<\frac{1}{2}.
\end{equation}
Then, for any initial condition $ x(0)\in\calx $, the  \emph{static feedback} intervention 
\begin{equation}\label{u-static-fb}
u(t)=\proj_\calu\big(aP^\top x(t)\big),
\end{equation}
steers the action profile $ x(t) $ to the social optimum $ x_\mathrm{opt} $.
\end{proposition}

\begin{proof}
The closed-loop dynamics of \eqref{ne-alg} and \eqref{u-static-fb} is
\begin{equation}\label{e:cl-static-fb}
\dot{x}=\Pi_{\calx}\big(x,-T(x)\big),
\end{equation}
with $ T(x)=F(x)-\proj_\calu(aP^\top x) $.	For the map $ T $, we have
\begin{multline}\label{e:t-hypmon-def}
(x-y)^\top\big(T(x)-T(y)\big)=\|x-y\|_{(I-\frac{1}{2}a(P+P^\top))}^2\\
-(x-y)^\top \big(\proj_\calu(aP^\top x)-\proj_\calu(aP^\top y)\big),
\end{multline}
for all $ x,y\in\R^n $. Note that the projection operator is nonexpansive \cite[Prop. 2.1.3]{bertsekas99nonlinear}, we thus deduce that the second term on the right-hand side of the above relation satisfies 
\begin{equation*}
(x-y)^\top \big(\proj_\calu(aP^\top x)-\proj_\calu(aP^\top y)\big)\leq \|aP\|\|x-y\|^2.
\end{equation*}
As a result, it follows from \eqref{e:t-hypmon-def} that
\begin{equation}\label{e:t-hypmon-final}
(x-y)^\top\big(T(x)-T(y)\big)\geq (1-2\|aP\|)\|x-y\|^2,\quad \forall x,y\in\R^n.
\end{equation}
This means that $ T $ is hypomonotone, hence the dynamics \eqref{e:cl-static-fb} has a unique solution $ x(t) $ for all $ t\geq 0 $ \cite[Thm. 1]{brogliato2006equivalence}.

Let $ \tilde x:=x-x_\mathrm{opt} $, and consider the Lyapunov candidate $ V(\tilde x)=\frac{1}{2}\|\tilde{x}\|^2 $. The time-derivative of the evolution of $ V $ along the solution of \eqref{e:cl-static-fb} satisfies
\begin{multline*}
\nabla V(\tilde{x})^\top \Pi_{\calx}\big(x,-T(x)\big)=-\tilde{x}^\top T(x)\\
- \tilde{x}^\top\proj_{\caln_{\calx}(x)}\big(-T(x)\big),
\end{multline*}
where we have used Moreau’s decomposition theorem. Recall that $ - \tilde{x}^\top\proj_{\caln_{\calx}(x)}\big(-T(x)\big)\leq 0 $ since $ x,x_\mathrm{opt}\in\calx $. We therefore have
\begin{align}\label{e:v-dot-T}
\nabla V(\tilde{x})^\top \Pi_{\calx}\big(x,-T(x)\big)\leq -\tilde{x}^\top T(x).
\end{align}
Note from \eqref{x-opt} that $ (y-x_\mathrm{opt})^\top H(x_\mathrm{opt})\geq 0$ for all $  y\in\calx. $
Adding the left-hand side of the this inequality evaluated at $ y=x $ to the right-hand side of \eqref{e:v-dot-T} yields
\begin{multline}\label{e:v-dot-state-fb}
\nabla V(\tilde{x})^\top \Pi_{\calx}\big(x,-T(x)\big)\leq  -\|\tilde{x}\|_{(I-\frac{1}{2}a(P+P^\top))}^2\\
+\tilde{x}^\top \big(\proj_\calu(aP^\top x)-aP^\top x_\mathrm{opt}\big),
\end{multline}
where the definitions of $ T $ and $ H $ are used. The relation $ aP^\top x_\mathrm{opt}\in\calu $ implies that $ aP^\top x_\mathrm{opt}=\proj_\calu(aP^\top x_\mathrm{opt}) $. This together with \eqref{e:t-hypmon-def} evaluated at $y=x_\mathrm{opt}$ means that the right-hand side of \eqref{e:v-dot-state-fb} is equal to $ -\tilde{x}^\top(T(x)-T(x_\mathrm{opt})) $.
We therefore deduce from  \eqref{e:t-hypmon-final} that
\begin{equation*}
\nabla V(\tilde{x})^\top \Pi_{\calx}\big(x,-T(x)\big)\leq -(1-2\|aP\|)\|\tilde{x}\|^2.
\end{equation*}
It then follows from \eqref{ap-weak-coupling} that $ V $ decreases monotonically along the solution of the closed-loop dynamics and the action profile $ x(t) $ converges to $ x_\mathrm{opt} $.
\end{proof}

Based on Proposition \ref{prop:static-fb}, the static feedback intervention \eqref{u-static-fb} steers the actions of the players to the social optimum under the condition \eqref{ap-weak-coupling}. 
Interestingly, this condition can be dropped in the case where the constraint set $ \calu $ is sufficiently ``large'', namely if $ aP^\top \bar x\in  \calu $ for all $\bar x\in \calx$; a trivial example is given by $\calu=\R^n$.
The following corollary summarizes this argument.
\begin{corollary}\label{cor:state-fb}
Consider  the pseudo-gradient dynamics \eqref{ne-alg}, and	assume for all $ \bar x\in\calx $, we have $ aP^\top \bar x\in  \calu $.
Then, for any initial condition $ x(0)\in\calx $, the  static feedback intervention 
\begin{equation}\label{u-static-fb-no-const}
u(t)=a P^\top x(t),
\end{equation}
steers the action profile $ x(t) $ to the social optimum $ x_\mathrm{opt} $.
\end{corollary}
\begin{proof}
Note that the state feedback intervention \eqref{u-static-fb} is equivalent to \eqref{u-static-fb-no-const} as $ aP^\top x\in  \calu $, that is $ \proj_\calu(aP^\top x)=a P^\top x $. We therefore deduce from the proof of Proposition \ref{prop:static-fb} that the closed-loop system has a unique solution $ x(t) $ for all $ t\geq 0 $. Moreover, given the Lyapunov candidate $ V(\tilde x)=\frac{1}{2}\|\tilde{x}\|^2 $ with $ \tilde x=x-x_\mathrm{opt} $, its time-derivative along $ x(t) $ satisfies \eqref{e:v-dot-state-fb}. Next we use $ \proj_\calu(aP^\top x)=a P^\top x $ and rewrite \eqref{e:v-dot-state-fb} as follows:
\begin{equation*}
\nabla V(\tilde{x})^\top \Pi_{\calx}\big(x,-T(x)\big)\leq  -\|\tilde{x}\|_{(I-a(P+P^\top))}^2.
\end{equation*}
We conclude from  $ I-a(P+P^\top)\succ \bze $ (cf. Assumption \ref{asm:opt}) that $ V $ decreases monotonically along the solution of the closed-loop dynamics and $ x(t) $ converges to $ x_\mathrm{opt} $.
\end{proof}

\begin{remark}\label{rem:potential}
It is worth mentioning that modifying the standalone marginal returns in \eqref{u-pay-def} by setting $u_i=a\sum_{j\in\cali}P_{ji}x_j$, transforms 
the network game into a ``potential game'' \footnote{A game $ G=(\cali, (U_i)_{i\in\cali}, (\calx_i)_{i\in\cali}) $ is an (exact) potential game if there exists a potential function $ \Phi:\calx\to \R $ such that $ U_i(x_i,x_{-i})- U_i(y_i,x_{-i})=\Phi(x_i,x_{-i})- \Phi(y_i,x_{-i})$ for all $ x_i,y_i\in\calx_i $, $ x_{-i}\in\prod_{j\neq i}\calx_j $, and $ i\in\cali $ \cite{monderer1996potential}.} with the potential function being the social welfare, namely   $ \sum_{i \in \cali} W_i\big(x_i,z_i(x)\big) $. 	
In fact, bearing in mind that $P_{ji}$ reflects the influence of player $i$ on player $j$, the aforementioned modification balances the game such that the mutual effects between any pair of players become identical. The protocol \eqref{u-static-fb-no-const} provides a dynamic counterpart of this marginal returns modification.
\oprocend
\end{remark}

\subsection{Dynamic intervention with estimated social optimum}
Next we consider the scenario where the regulator  is not aware of the game information $(aP, b)$ and $\calx_i$'s, but instead has a reliable estimate of the social optimum $x_\mathrm{opt}$, namely $ x_s\in\cals $. In this case, the regulator can resort to an integral control-based intervention to obtain convergence of the action profile to $ x_s$. We present such intervention and its convergence guarantees in the following proposition:

\begin{proposition}\label{prop:dynamic}
Consider the pseudo-gradient dynamics \eqref{ne-alg}.
Let $x_s \in \cals$ and consider  
the  \emph{dynamic} intervention 
\begin{equation}\label{u-dynamic}
\dot{u}(t)=\Pi_\calu\big(u(t),x_s-x(t)\big).
\end{equation}
Then, for any initial condition $ x(0)\in\calx $,
the above intervention protocol steers the action profile $ x(t) $ to the point $ x_s $.
\end{proposition}

\begin{proof}
By using \eqref{ne-alg} and \eqref{u-dynamic}, the dynamics of the overall closed-loop system is given by
\begin{equation}\label{e:cl-dynm-int}
\dot{\xi}=\Pi_\Lambda\big(\xi,-T(\xi)\big),
\end{equation}
where $ \xi=\col(x,u) $,  $\Lambda=\calx\times \calu$ and
\begin{equation*}
T(\xi)=\begin{bmatrix}
	F(x)-u\\ x-x_s
\end{bmatrix}.
\end{equation*}
We deduce from  strong monotonicity of $F$  (see \eqref{F-strong-monotone}) that the above mapping is monotone, and  the set $ \Lambda $ is closed and convex. We then obtain from \cite[Thm. 1]{brogliato2006equivalence} that, for any initial condition $ \xi(0)\in\Lambda $, the dynamics \eqref{e:cl-dynm-int} admits a unique solution $ \xi(t) $ for all $ t\geq 0 $.

It follows form $ x_s\in\cals $ that there exists a $ u_s\in\calu $ such that
\begin{equation}\label{vi-ineq-us}
(y-x_s)^\top\big( F(x_s)-u_s \big)\geq 0,\quad \forall y\in\calx.
\end{equation}
Next we use the inequality above and prove that $ (x(t),u(t)) $ converges to $ (x_s,u_s) $. To this end, 
consider the 	Lyapunov candidate $ V(\xi)=\frac{1}{2}\|\tilde{x}\|^2+\frac{1}{2}\|\tilde{u}\|^2 $ with $ \tilde{x}=x-x_s $ and $ \tilde{u}=u-u_s $. The time-derivative of the evolution of $ V $ along the solution of \eqref{e:cl-dynm-int} satisfies
\begin{multline*}
\nabla V(\xi)^\top \Pi_\Lambda\big(\xi,-T(\xi)\big)=\tilde{x}^\top \big(-F(x)+u\big)-\tilde{u}^\top \tilde{x}\\- \tilde{x}^\top \proj_{\caln_{\calx}(x)}\big(-F(x)+u\big)
-\tilde{u}^\top \proj_{\caln_{\calu}(u)}(-\tilde{x}),
\end{multline*}
where we have used the Moreau’s decomposition theorem.
Since $x,x_s\in\calx$ and $u,u_s\in\calu$, we have  $ - \tilde{x}^\top \proj_{\caln_{\calx}(x)}\big(-F(x)+u\big)\leq 0$ and $ -\tilde{u}^\top \proj_{\caln_{\calu}(u)}(-\tilde{x})\leq 0$, respectively. We then obtain that
\begin{align*}
\nabla V(\xi)^\top \Pi_\Lambda\big(\xi,-T(\xi)\big)	\leq \tilde{x}^\top \big(-F(x)+u\big)-\tilde{u}^\top \tilde{x}.
\end{align*}
Now we add the left-hand side of \eqref{vi-ineq-us} evaluated at $ y=x $
to the right-hand side of the foregoing inequality to get
\begin{align*}
\nabla V(\xi)^\top \Pi_\Lambda\big(\xi,-T(\xi)\big)	&\leq -\tilde{x}^\top \big(F(x)-F(x_s)\big)\\
&=-\|\tilde{x}\|_{(I-\frac{1}{2}a(P+P^\top))}^2 ,
\end{align*}
where the equality follows from the definition of $ F $ given by \eqref{e:f-map}. Note that $ I-\frac{1}{2}a(P+P^\top)\succeq \frac{1}{2}I $ as a consequence of Assumption \ref{asm:opt}, hence we deduce that
\begin{equation*}
\nabla V(\xi)^\top \Pi_\Lambda\big(\xi,-T(\xi)\big)	\leq -\frac{1}{2}\|\tilde{x}\|^2.
\end{equation*}
Let  $\xi_0\in \Lambda$, and $ \xi(t) $ be a solution starting from the initial condition $ \xi(0)=\xi_0 $. Moreover, 
let  $ \delta:=V(\xi_0) $ and define the set $ \Omega:=\{\xi\in \Lambda  \mid V(\xi)\leq \delta\} $. Note that $ \xi(0)\in \Omega $, and $ \Omega $ is compact since $V(\xi) \to \infty  $ as $ \|\xi\|\to\infty $. It also follows from the inequality above that the solution $ \xi(t) $ remains in $ \Omega $. We then use the invariance principle for discontinuous systems \cite[Prop. 2.1]{cherukuri2016asymptotic} to conclude that the solution of the closed-loop system  converges to the largest invariant set contained in $ \{\xi\in \Omega \mid \nabla V(\xi)^\top \Pi_\Lambda\big(\xi,-T(\xi)\big)=0 \} $. This together with the inequality above imply that $ \xi(t) $ also converges to the largest invariant set in $ \{\xi\in \Omega \mid \tilde{x}=\bze \} $. We therefore conclude that for any initial condition $ \xi(0)\in  \Lambda$, the action profile $ x(t) $  converges to $ x_s $, and this completes the proof.
\end{proof}

\subsection{Adaptive intervention with known standalone marginal returns}
Recall that in case the regulator knows $ aP $ or the social optimum $ x_\mathrm{opt} $, she can steer the players to the social optimum by implementing the previously discussed interventions. Here, we shift our focus to the case where both $aP$ and $x_\mathrm{opt}$ are unknown to the regulator, and she merely has knowledge about the individual standalone marginal returns of the players $b$. It turns out that such limited knowledge  substantially complicates the problem faced by the regulator. To partially tame this complexity, we restrict our attention in this subsection to the case of unconstrained actions and interventions, i.e., $ \calx_i=\R $ and $ \calu=\R^n $, and undirected networks, i.e. $ P=P^\top $. As a result, the pseudo-gradient dynamics \eqref{ne-alg} simplifies to the following:
\begin{equation}\label{ne-alg-unconst}
\dot{x}(t)=(-I+aP)x(t)+b+u(t).
\end{equation}

A natural approach to tackle this problem is to resort to adaptive control techniques which potentially allow to compensate for lack of complete knowledge on the system dynamics. However, there are certain obstacles that hinder an application of standard adaptive control schemes. First, a control design based on the regulation error $ x(t)-x_\mathrm{opt} $ is not feasible since $ x_\mathrm{opt} $ is unknown. A second attempt would be to try to estimate $x_\mathrm{opt}$ by using a reference model such as $\dot{x}_m(t)=(-I+2aP)x_m(t)+b$. However, while $x_m(t)$ converges to $ x_\mathrm{opt} $ (see Corollary \ref{cor:state-fb} with $P=P^T$), the reference model is not implementable as the network matrix $aP$ is unknown. 

To overcome these challenges, we propose the  \emph{adaptive feedback} intervention protocol 
\begin{equation}\label{u-adap}
u(t)=K(t) x(t)
\end{equation}
with an adaptive gain matrix $ K(t) $ determined by the following extended nonlinear dynamics:

\begin{subequations}\label{adap}
\begin{align}
\dot{z}(t)&=-z(t)+K(t)x(t)+b+u(t),\label{adap1}\\
\dot{w}(t)&=-w(t)+e(t)x^\top(t)x(t),\label{adap2}\\
\dot{K}(t)&=e(t)x^\top(t),\label{adap3}
\end{align}
\end{subequations} 
where
\begin{equation*}
e(t):=x(t)-z(t)-w(t).
\end{equation*}
Note that the intervention only uses information on $ b $, and no knowledge on  $ aP $ or $ x_\mathrm{opt} $ is required. The first dynamics \eqref{adap1} aims to replicate the pseudo-gradient dynamics \eqref{ne-alg-unconst} and generate $ z(t) $ such that it tracks the action profile  $ x(t) $.  The second dynamics \eqref{adap2} is included for technical reasons and is needed to guarantee boundedness of all solutions. The last dynamics \eqref{adap3} is chosen such that sign-indefinite terms in the time-derivative	of the Lyapunov function are canceled out. As a result, all solutions of the closed-loop system are bounded as stated in the following lemma. 
\begin{lemma}\label{lem:adap-bound}
Consider the pseudo-gradient dynamics \eqref{ne-alg-unconst} and let $P=P^\top$.
Then, under the  adaptive feedback intervention 
given by \eqref{u-adap} and \eqref{adap}, all solutions of the closed-loop system are  bounded. 
\end{lemma}

\begin{proof}
See the appendix.
\end{proof}

The next result establishes convergence to the social optimum  $ x_\mathrm{opt} $.
\begin{theorem}
Let $P=P^\top$ and consider the pseudo-gradient dynamics \eqref{ne-alg-unconst} interconnected with the adaptive feedback intervention given by \eqref{u-adap} and \eqref{adap}.
Then,	the action profile $ x(t) $  converges to the social optimum $ x_\mathrm{opt} $. 
\end{theorem}

\begin{proof}

Let $ \xi:=(x,e,\Psi) $ with $ \Psi=K-aP $. Then, bearing in mind \eqref{ne-alg-unconst}, \eqref{u-adap} and \eqref{adap}, $\xi$ admits the following dynamics
\begin{subequations}\label{incr-adap-thm}
\begin{align}
	\dot{x}&=(-I+2aP)x+b+\Psi x, \label{x-psi-relat-thm}\\
	\dot{e}&=-e-\Psi x-ex^\top x,\label{incr-adap1-thm}\\
	\dot \Psi&= e x^\top.\label{incr-adap2-thm}
\end{align}
\end{subequations}
We proceed by following similar arguments as in the proof of the LaSalle's invariance principle \cite[Thm. 4.4]{khalil2002nonlinear}, but the proof is tailored for a single (yet arbitrary) trajectory.
Let  $\xi_0:=(x_0,e_0,\Psi_0)$ with some $ x_0,e_0\in\R^n $ and $ \Psi_0\in \R^{n\times n} $, and $ \xi(t) $ be a solution starting from the initial condition $ \xi(0)=\xi_0 $. It follows from Lemma \ref{lem:adap-bound} that this solution is bounded. 
Thus, there exists a compact set $\cal{D}$ such that $\xi(t)\in \cal{D}$ for all $t\geq 0$. It also follows from \cite[Lem. 4.1]{khalil2002nonlinear} that the positive limit set $ \Omega $ of $ \xi(t) $ is  nonempty, compact, and invariant. Moreover, $ \xi(t) $ approaches $ \Omega $ as $ t $ tends to infinity.

We now consider  the function  
$$ V(\xi):= \frac{1}{2}\|e\|^2+\frac{1}{2}\|\Psi\|_{\mathrm{F}}^2,$$
where we recall that $\|\Psi\|_{\mathrm{F}}$ is the Frobenius norm. 
The derivative of $ V $ along the solutions of \eqref{incr-adap-thm} is
\begin{equation}\label{vdot-adap-thm}
\begin{split}
	\dot{V}&=-\|e\|^2-e^\top\Psi x-\|e\|^2 \|x\|^2+\Tr(\Psi^\top e x^\top)\\
	&=-\|e\|^2-\|e\|^2 \|x\|^2,
\end{split}
\end{equation}
where the last equality is obtained using $ e^\top\Psi x=\Tr(\Psi^\top e x^\top) $. Therefore, we have  $ V\geq 0 $ and $ \dot{V}\leq 0 $ which 
implies that $ V(\xi(t)) $ has a limit $ V_\infty\geq 0 $ as $ t\to \infty $. Pick any point $ \xi'\in \Omega $, then there is a sequence $ \{t_n\} $, with $ t_n\to\infty $ as $ n\to\infty $, such that $ \xi(t_n)\to \xi' $ as $ n\to\infty $. We  obtain from continuity of $ V $ that $ V(\xi')=\lim_{n\to\infty}V(\xi(t_n))= V_\infty$. Therefore, since $ \xi' $ is chosen arbitrary, we deduce  that $ V(\xi)=V_\infty $ for all $\xi \in \Omega $, which means that on the invariant set $ \Omega $, the function $ V $ is constant. Moreover, we have $ \dot{V}(\xi(t))=0 $ for all $\xi(t)\in \Omega$. Let $ E:=\{\xi\in{\cal{D}}\mid \dot{V}(\xi)=0\} $, then we have $ \Omega \subset E $. Now let $ M $ be the largest invariant set inside $ E $, subsequently we have the following relations
\begin{equation*}
\Omega \subset M \subset E \subset \mathcal{D}.
\end{equation*}
Noting that $ \xi(t) $ approaches $ \Omega  $ as $ t\to \infty $, we obtain that $ \xi(t) $ approaches $ M $ as $ t\to\infty $.

The last step is to find the set $ M $. Note from the definition of $ E $ and \eqref{vdot-adap-thm} that $ E = \{\xi\in\mathcal{D}\mid e=\bze\} $. Thus, on the invariant set $M$, the dynamics of \eqref{incr-adap-thm} 
reads as
\begin{align*}
\dot{x}&=(-I+2aP)x+b,\\
\bze &=-\Psi x,\\
\dot \Psi&=\bze.
\end{align*}
Noting that $-I+2aP$ is Hurwitz as a consequence of Assumption \ref{asm:opt},
the largest invariant set in $E$ is given by
\begin{equation*}
M=\big\{\xi\in\mathcal{D}\mid x=x_\mathrm{opt},\ e=\bze,\ \Psi x_\mathrm{opt}=\bze \big\}.
\end{equation*}
Consequently, we conclude that $ x(t) $ converges to $ x_\mathrm{opt} $ as desired.
\end{proof}

\section{Illustrative examples}\label{sec:sim}
We consider a Cournot competition where a set of $ \cali=\{1,\dots,10\} $ firms produce differentiated goods \cite{bramoulle2014strategic}. For each firm $ i $, we denote the amount of good by $ x_i\in\calx_i$, and its corresponding price is obtained from the  inverse demand function 
$
p_i(x) = \alpha_i - \frac{1}{2}\left(x_i + 2 \beta \sum_{j \neq i}P_{ij} x_j\right).
$
In this equation, $ \alpha_i>0 $ is the maximum price that consumers would pay for the good, $ \beta P_{ij}\geq 0 $ is the degree of product substitutability, where
$ P_{ij}\in (0,1] $ if the product of firm $ j $ is a substitute for firm $ i $ and $ P_{ij}=0 $ otherwise.\footnote{This is slightly different from \cite{bramoulle2014strategic} which considers  $ P_{ij}=P_{ji}\in \{0,1\} $.} 
The payoff function of firm $ i $, therefore, can be written in the form of \eqref{u-pay-def}, \eqref{w-pay-def} as follows
\begin{align*}
U_i(x_i,x,u_i)&=x_i p_i(x) - x_i d_i + x_i u_i\\
&=- \frac{1}{2}x_i^2+x_i\Big(a \sum_{j \neq i}P_{ij} x_j+b_i\Big)+ x_i u_i,
\end{align*}
where $ a=-\beta $ and $ b_i=\alpha_i-d_i $ with $ d_i>0 $ being the marginal cost and $u_i$ reflects taxes or subsidies provided by the regulator. 

Next we present the simulation results under our interventions and illustrate convergence of the players' actions to the social optimum.

\subsection{Open-loop, static feedback and dynamic interventions}
Here, we 
consider a competition where  $ x_i\in\calx_i=\R_{\geq 0} $, $ \beta = 0.2 $, and the products of the firms are substitutable according to the weighted directed graph depicted in Fig. \ref{fig:directed-graph}. In this graph, the weight of each link from firm $j$ to firm $i$ denotes the weight $P_{ij}$, and the number next to each node $ i $ indicates its  standalone marginal return, namely $ \alpha_i-d_i $.

The social optimum of this game is 
\begin{equation*}
x_\mathrm{opt}=\col(2.19,
0.01,
0.99,
0.49,
1.34,
3.4,
0,
0,
0.99,
0.04).
\end{equation*}
The regulator incentivizes the firms to the social optimum by applying bounded taxes $ u\in\calu $ where $ \calu = [-2,0]^{10} $. Next we use intervention mechanisms to obtain suitable taxes.

\begin{figure}
\centering
\includegraphics[width=\columnwidth]{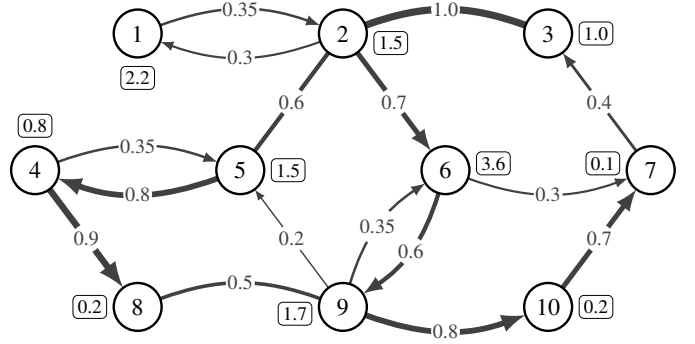}
\caption{The directed network illustrating asymmetrical product substitutability.}
\label{fig:directed-graph}
\end{figure}

\bigskip

\noindent\emph{Open-loop intervention:} Having full information of the game, the regulator can find the social optimum given above as well as $ \caln_{\calx}(x_\mathrm{opt}) $. 
It then follows from \eqref{u-static} that
$
u_\mathrm{opt}=-\col(0.001,\allowbreak
0.96,\allowbreak
0.003,\allowbreak
0.09,\allowbreak
0.08,\allowbreak
0.11,\allowbreak
0,\allowbreak
0.01,\allowbreak
0.29,\allowbreak
0).
$
The simulation results under the intervention $ u(t)\equiv u_\mathrm{opt}$ with the pseudo-gradient dynamics \eqref{ne-alg}  and initialized arbitrarily are shown in Fig. \ref{fig:static_openloop}, which demonstrates convergence of the players' actions to the social optimum.

\begin{figure}
\centering
\includegraphics[width=\columnwidth]{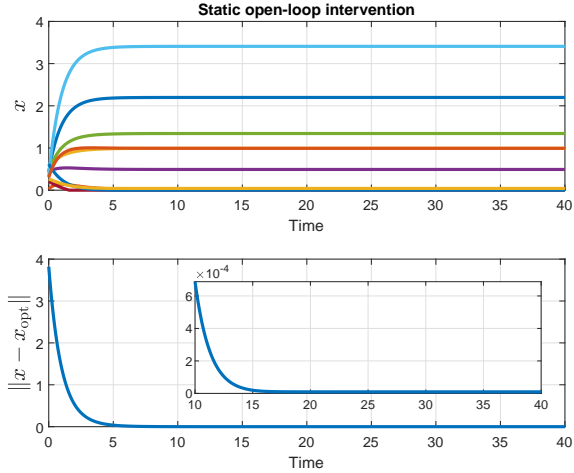}
\caption{Actions of the players and their distance to social optimum under static open-loop intervention.}
\label{fig:static_openloop}
\end{figure}

\bigskip

\noindent\emph{Static feedback intervention:} Under the assumption that the regulator knows $ aP $, she can implement \eqref{u-static-fb}. The regulator can, therefore, steer the actions of the players to the social optimum as shown in Fig. \ref{fig:static_feedback}.

\begin{figure}
\centering
\includegraphics[width=\columnwidth]{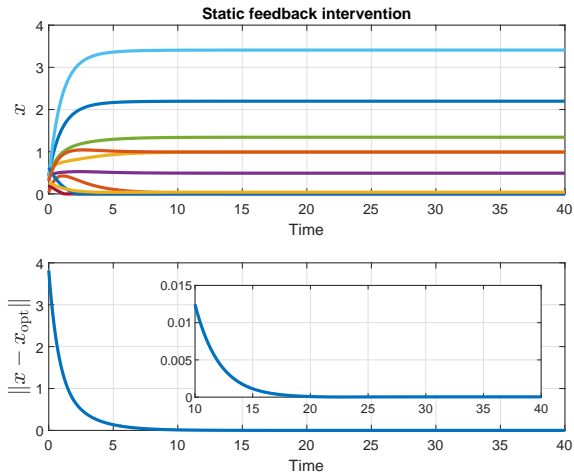}
\caption{Actions of the players and their distance to social optimum under static feedback intervention.}
\label{fig:static_feedback}
\end{figure}

\bigskip

\noindent\emph{Dynamic intervention:} We assume that the regulator only knows the value of the social optimum, she can then implement \eqref{u-dynamic} with $ x_s = x_\mathrm{opt} $. As a result, the players' action profile converges to $ x_\mathrm{opt} $ as desired, and this is illustrated in Fig. \ref{fig:dynamic}.

\begin{figure}
\centering
\includegraphics[width=\columnwidth]{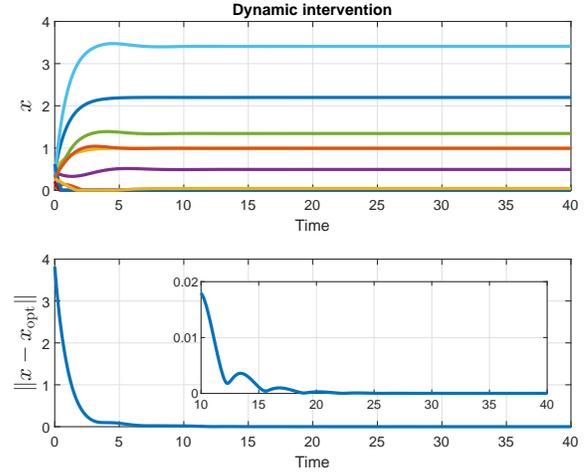}
\caption{Actions of the players and their distance to social optimum under dynamic intervention.}
\label{fig:dynamic}
\end{figure}

\subsection{Adaptive intervention}
In order to demonstrate performance of the adaptive intervention, we consider the case $ \calx_i=\R $ and assume that product substitutability is represented by the undirected graph in Fig. \ref{fig:undirected-graph}, where the weight of a link between firms $j$ and $i$ denotes the value $P_{ij}=P_{ji}$.
The actions of the firms are obtained from the pseudo-gradient dynamics \eqref{ne-alg-unconst} with arbitrary initial conditions.

\begin{figure}
\centering
\includegraphics[width=\columnwidth]{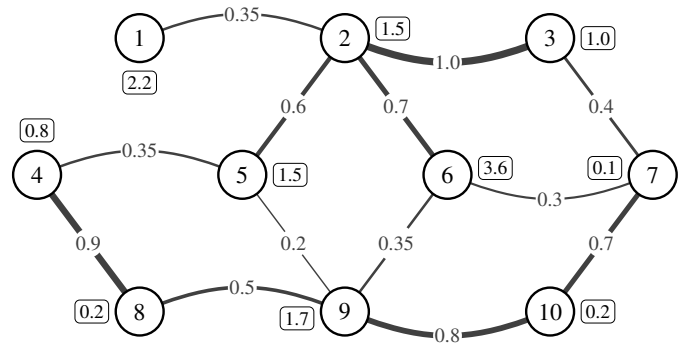}
\caption{The undirected network illustrating symmetrical product substitutability.}
\label{fig:undirected-graph}
\end{figure}

For this game, we obtain $
x_\mathrm{opt}=\col(2.31,\allowbreak
-0.79,\allowbreak
1.41,\allowbreak
0.68,\allowbreak
1.50,\allowbreak
3.73,\allowbreak
-0.57,\allowbreak
-0.26,\allowbreak
1.10,\allowbreak
0.006).$ To steer the players to this point, the regulator applies the adaptive intervention \eqref{u-adap}. Note that here $ u\in\R^n $, which represents taxes and subsidies. Fig. \ref{fig:adaptive} shows the actions of the players under this intervention and demonstrates convergence to the social optimum.

\begin{figure}
\centering
\includegraphics[width=\columnwidth]{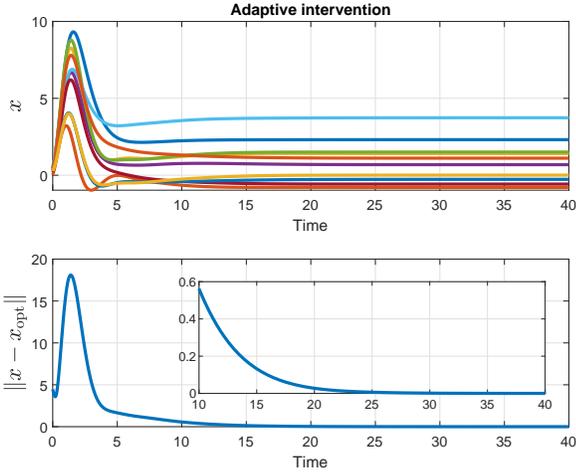}
\caption{Actions of the players and their distance to social optimum under adaptive intervention.}
\label{fig:adaptive}
\end{figure}

\section{Conclusions}\label{sec-conclusion}
We have proposed intervention protocols that are able to alter the outcome of noncooperative network games toward the social optimum. 
We investigated different sets of information available to the regulator and proposed intervention mechanisms tailored to each case. Convergence to the maximizer of the social welfare function is analytically shown for all the proposed  mechanisms, and the efficiency of the proposed protocols is demonstrated on a numerical case study of Cournot competition with differentiated goods.
Future works include extension of the results to  network games with general payoff functions and to the case where the actions of the  players share coupled constraints.

\bibliographystyle{IEEEtran}
\bibliography{MyReferences}

\section*{Appendix: proofs of the technical lemmas}
\noindent\textit{Proof of Lemma \ref{lem:soci-uniq}.}
The inequality 
\eqref{e:iff-sc} is equivalent to the matrix inequality  $ -I+a(P+P^\top) \prec \bze $ and thus to strong concavity of  $x \mapsto \sum_{i \in \cali} W_i\big(x_i,z_i(x)\big) $. The latter map admits at most one maximizer over the closed convex set $ \calx $ \cite[Prop. 2.1.1]{bertsekas99nonlinear}. To show the existence of such unique maximizer,  pick a point $p \in \mathcal{X}$ and define the following set:
$$ \caly:=\left\{y\in\calx \mid \sum_{i \in \cali} W_i\big(p_i,z_i(p)\big) \leq \sum_{i \in \cali} W_i\big(y_i,z_i(y)\big)\right\}.$$
The set $ \caly $ is compact as a result of strong concavity of the social welfare function, and   the maximization problem \eqref{social} is equivalent to
\begin{equation*}
x_\mathrm{opt}\in \argmax_{y\in\caly}\sum_{i\in\cali} W_i\big(y_i,z_i(y)\big).
\end{equation*}
The existence of  $ x_\mathrm{opt} $ then follows from  Weierstrass’ Theorem \cite[Prop. A.8]{bertsekas99nonlinear}, and this concludes the proof.
\squarend

\medskip{}
\noindent\textit{Proof of Lemma \ref{lem:adap-bound}.}
We divide the proof into three parts, 
that include 1) 
proving that certain signals of the overall closed-loop system are $ \call_\infty $ and $ \call_2 $; 2) upper bounding all the closed-loop state variables by a common signal, denoted by $\ell(t)$; and 3) showing that $\ell(t)$ and thus all the state-variables are bounded.

\medskip

\noindent\emph{Step 1 ($ \call_\infty $ and $ \call_2 $  analysis):} We start our proof by analyzing evolution of $ (e,\Psi) $, where  $ \Psi:=K-aP $. It follows from \eqref{ne-alg-unconst}, \eqref{u-adap}, and  \eqref{adap} that
\begin{subequations}\label{incr-adap}
\begin{align}
\dot{e}&=-e-\Psi x-ex^\top x,\label{incr-adap1}\\
\dot \Psi&= e x^\top.\label{incr-adap2}
\end{align}
\end{subequations}
Consider the Lyapunov candidate
\begin{equation}\label{v-def}
V(e,\Psi):=\frac{1}{2}\|e\|^2+\frac{1}{2}\|\Psi\|_{\mathrm{F}}^2,
\end{equation}
where we recall that $\|\Psi\|_{\mathrm{F}}$ is the Frobenius norm. 
The derivative of $ V $ along the solutions of \eqref{incr-adap} is
\begin{equation}\label{vdot-adap}
\begin{split}
\dot{V}&=-\|e\|^2-e^\top\Psi x-\|e\|^2 \|x\|^2+\Tr(\Psi^\top e x^\top)\\
&=-\|e\|^2-\|e\|^2 \|x\|^2,
\end{split}
\end{equation}
where the last equality is obtained using $ e^\top\Psi x=\Tr(\Psi^\top e x^\top) $. Therefore, we have  $ V\geq 0 $ and $ \dot{V}\leq 0 $ which results in
\begin{equation}\label{v-inf}
V_\infty:=\lim_{t\to\infty}V\big(e(t),\Psi(t)\big)\leq V\big(e(0),\Psi(0)\big).
\end{equation}
Thus, we obtain $ e,\Psi \in \call_\infty $. We proceed to show that the closed-loop signals $\dot{\Psi}$, $e$, $e\|x\|$ belong to $ \call_2 $, for any $ x $. Note  from \eqref{vdot-adap} that 
\begin{equation*}
\dot{V}\leq -\|e\|^2.
\end{equation*}
We integrate both sides of the inequality above and use \eqref{v-inf} to get $$ \int_{0}^{\infty}\|e(\tau)\|^2 d\tau\leq V\big(e(0),\Psi(0)\big)-V_\infty<\infty. $$
Consequently, we have $ e\in \call_2 $. Moreover, we deduce from an analogous analysis for $ e\|x\| $ in \eqref{vdot-adap} that $ e\|x\|\in \call_2 $. Now we rewrite the dynamics of $ \Psi $ given in  \eqref{incr-adap2} as follows
\begin{equation}\label{psi-dot-l2}
\dot{\Psi}=e(1+\|x\|)\frac{x^\top}{1+\|x\|}.
\end{equation}
Note that for any $x$, we have $ x/(1+\|x\|)\in \call_\infty $, and $ e(1+\|x\|)\in \call_2 $ since $ e,e\|x\|\in\call_2 $. Thus we derive from \eqref{psi-dot-l2} that $ \dot{\Psi}\in \call_2 $.
We record below our findings in Step 1 of the proof for a later use: 
\begin{itemize}
\item $ \Psi,e \in \call_\infty $,\\
\item $ \dot{\Psi},e,e\|x\|\in\call_2 $.
\end{itemize}

\medskip

\noindent\emph{Step 2 (Determining a common upper bound):}
Consider a solution $ (x(t),z(t),w(t),K(t)) $ of the closed-loop system, made of \eqref{ne-alg-unconst}, \eqref{u-adap}, and \eqref{adap},  starting at an arbitrary initial condition. 
Note that $ K(t) $ is bounded as $ \Psi(t)\in\call_\infty $. Next, we find a common upper bound for the closed-loop signals $ (x(t),z(t),w(t)) $ using the properties established in the previous step. This will allow us to prove boundedness of the all the closed-loop signals in Step 3. 

We proceed the analysis by introducing the following normalizing signal 
\begin{equation}\label{l-def}
\ell(t):=\sqrt{1+\|x(t)\|_{2\delta}^2},
\end{equation} 
where $ \|x(t)\|_{2\delta} $ is the exponentially weighted $ \call_2 $ norm of $ x $ defined as 
\begin{equation*}
\|x(t)\|_{2\delta}:=\Big(\int_0^t \exp\big(-\delta(t-\tau)\big)x^\top(\tau)x(\tau)d\tau\Big)^{\frac{1}{2}}
\end{equation*}
for a given $\delta \geq 0$.
Next we show that the closed-loop signals $(x,z,w)$ can be bounded from above by an affine function of  $ \ell $. Noting $ u=Kx $ and $ \Psi=K-aP $, we rewrite  \eqref{ne-alg-unconst} as 
\begin{equation}\label{x-psi-relat}
\dot{x}=(-I+2aP)x+b+\Psi x.
\end{equation}
Note that $ (-I+2aP) $ is Hurwitz as a consequence of Assumption \ref{asm:opt}, thus there exist constants $ k_0,\alpha_0>0 $ that satisfy 
\begin{equation}\label{transition-ineq}
\big\|\exp\big((-I+2aP)(t-\tau)\big)\big\|\leq k_0 \exp\big(-\alpha_0(t-\tau)\big),
\end{equation}
for all $ \tau\in [0,t] $.	It then follows from \eqref{x-psi-relat}, the established property  $ \Psi\in\call_\infty $, and \cite[Lem. 3.3.3(i)]{ioannou2012robust} that for any given $ \delta\in [0,2\alpha_0) $, there exist constants $ c_0,c_1>0 $ such that  $ \|x\|\leq c_0+c_1 \|x\|_{2\delta}$. Similarly, we  obtain from \eqref{adap1}   that for any $ \delta\in [0,2) $, we have $ \|z\|\leq c_2+c_3 \|x\|_{2\delta}$  for some $ c_2,c_3>0 $. Regarding $ w $, we employ $ e\in\call_\infty $ together with the definition of $ e $ and the upper bounds on $ \|x\| $ and $ \|z\| $ to deduce that for any   $ \delta\in[0,2\min\{1,\alpha_0\}) $, there are $ c_4,c_5>0 $ such that $ \|w\|\leq c_4+c_5 \|x\|_{2\delta} $. Therefore, we can use $ \|x\|_{2\delta} $ and bound from above $ x,z,w $. Note from the definition of $ \ell $ that $ \|x\|_{2\delta}\leq \ell $.  Consequently, for any $ \delta\in[0,2\min\{1,\alpha_0\}) $, the followings hold:
\begin{equation}\label{bound-on-xzw}
\|x\|\leq c_0+c_1\, \ell,\quad \|z\|\leq c_2+c_3\, \ell,\quad \|w\|\leq c_4+c_5\, \ell.
\end{equation}
We see from the above relations that the signal $ \ell $ provides a common upper bound for all the closed-loop signals, thus  these signals are bounded provided that $ \ell $ is bounded.   
\bigskip

\noindent\emph{Step 3 (Boundedness analysis):}  Here, we address boundedness analysis of $ \ell $  using the fact that some of the signals belong to $ \call_2 $. We perform the analysis in two steps. First, we find an implicit upper bound of $ \ell $ which includes the $ \call_2 $ signals. We then use Bellman-Gronwall Lemma to find an explicit upper bound of $ \ell $  and conclude its boundedness.

In this part of the proof, we ease the notation by using $ \bm{c}>0 $ to denote  all positive constants whose actual values do not affect stability of the system. 
In other words, the forthcoming analysis is oblivious to the exact value of $\bm{c}$, and $\bm{c}$ is used merely for simplicity of the presentation. We remark that such notational convention is used in the classical textbook \cite{ioannou2006adaptive}.  

Bearing in mind the definition of $ \ell $ given by \eqref{l-def}, we use  \eqref{x-psi-relat} and \cite[Lem. 3.3.3(ii)]{ioannou2012robust} to infer that  for any $ \delta\in [0,2\alpha_0) $, we have
\begin{equation*}
\|x\|_{2\delta}\leq  \bm{c}+ \bm{c}\|\Psi x\|_{2\delta},
\end{equation*}
where $\alpha_0$ is defined in \eqref{transition-ineq}.
Thus we  obtain from \eqref{l-def} that 
\begin{equation}\label{l-bnd}
\ell^2\leq  \bm{c}+ \bm{c}\|\Psi x\|_{2\delta}^2.
\end{equation}
It follows from the inequality above that  the signal $ \ell$ is bounded from above  by the norm of $ \Psi x $. As a result, we proceed by analyzing $ \Psi x $, and for that, we introduce the following dynamics:
\begin{equation}\label{dyn-p-q}
\begin{alignedat}{2}
\dot{p}&=-\beta p+\dot{\Psi}x+\Psi \dot{x},\quad &&p(0)=\Psi(0)x(0),\\
\dot{q}&=-\beta q+\beta \Psi x,\quad &&q(0)=\bze,
\end{alignedat}
\end{equation}
where $ \beta>0 $. It is then straightforward to verify that $ \Psi x=p+q $. 
Moreover, note from \eqref{incr-adap1} that $ \Psi x=- \dot{e}-e-ex^\top x$. This, together with the dynamics of $ q $, allow us to write $ q=q_1+q_2-\beta e $  where
\begin{equation}\label{dyn-q1-q2}
\begin{alignedat}{2}
\dot{q_1}&=-\beta q_1-\beta (1-\beta)e,\quad &&q_1(0)=\beta e(0),\\
\dot{q_2}&=-\beta q_2-\beta ex^\top x,\quad &&q_2(0)=\bze.
\end{alignedat}
\end{equation}
As a result, we have obtained $ \Psi x=p+q_1+q_2-\beta e $. We next use this relation to find an upper bound of $ \|\Psi x\|_{2\delta} $ which explicitly depends on $ \beta $ and includes the $ \call_2 $ signals. Consequently, this provides us an upper bound for $ \ell $.
Towards this end, we consider the introduced dynamics \eqref{dyn-p-q} and \eqref{dyn-q1-q2} and obtain from  \cite[Lem. 3.3.3(ii)]{ioannou2012robust}  that for any $ \delta\in[0,\delta_1) $ where $ \delta_1\in(0,2\beta) $ is arbitrary, the following relations hold:
\begin{align*}
\|p\|_{2\delta}&\leq  \bm{c}+h(\beta) \, \|\dot{\Psi}x+\Psi \dot{x}\|_{2\delta},\\
\|q_1\|_{2\delta}&\leq  \bm{c}+h(\beta) \, \beta |1-\beta| \|e\|_{2\delta},\\
\|q_2\|_{2\delta}&\leq h(\beta) \, \beta \|ex^\top x\|_{2\delta},
\end{align*}
where 
\begin{equation*}
h(\beta):=\frac{1}{\sqrt{(\delta_1-\delta)(2\beta-\delta_1)}}.
\end{equation*}
The upper bound of $ \|q_2\|_{2\delta} $ does not have an additive constant term $ \bm{c} $ as we have $ q_2(0)=\bze $. Let $\beta > 1$ and $ \delta_1=1 $,  we then deduce that for any given $  \delta\in[0,1) $, the succeeding inequalities are satisfied for all  $ \beta>1 $:
\begin{align*}
\|p\|_{2\delta}&\leq  \bm{c}+ \bm{c}\,\beta^{-\frac{1}{2}} (\|\dot{\Psi}x\|_{2\delta}+\|\Psi \dot{x}\|_{2\delta}),\\
\|q_1\|_{2\delta}&\leq  \bm{c}+ \bm{c}\, \beta^{\frac{3}{2}}\|e\|_{2\delta},\\
\|q_2\|_{2\delta}&\leq  \bm{c}\, \beta^{\frac{1}{2}} \|ex^\top x\|_{2\delta},
\end{align*}
where we have used the triangular inequality to get the first relation. Note that  $  \bm{c} $ does not depend on $ \beta $, this property will be useful later in establishing boundedness of $\ell$. We now employ  the above inequalities in the relation $ \Psi x=p+q_1+q_2+\beta e $ to obtain 
\begin{equation}\label{psi-x-bnd}
\begin{split}
\|\Psi x\|_{2\delta}\leq  &\bm{c}+ \bm{c}\,\beta^{-\frac{1}{2}} (\|\dot{\Psi}x\|_{2\delta}+\|\Psi \dot{x}\|_{2\delta})
\\ &+ \bm{c}\, \beta^{\frac{3}{2}}\|e\|_{2\delta}+ \bm{c}\, \beta^{\frac{1}{2}} \|ex^\top x\|_{2\delta}+\beta \|e\|_{2\delta}.
\end{split}
\end{equation}
Next, we further bound the right-hand side of the inequality above by the $ \call_2 $ signals $ e\|x\|, \dot{\Psi}$ and the normalizing signal $ \ell $. 
For the term $  \beta^{-\frac{1}{2}}\|\dot{\Psi}x\|_{2\delta} $, we use the definition of the exponentially weighted $ \call_2 $ norm to get
\begin{equation}\label{bnd-dpsi-x}
{\small\begin{split}
\beta^{-1} \|\dot{\Psi}x\|_{2\delta}^2&\leq \beta^{-1}\int_{0}^{t}\!\!\exp\big(\!-\delta(t-\tau)\big)\|\dot{\Psi}(\tau)\|^2 \|x(\tau)\|^2d\tau\\
&\leq \beta^{-1}\int_{0}^{t}\!\!\exp\big(\!-\delta(t-\tau)\big)\|\dot{\Psi}(\tau)\|^2 \big( \bm{c}+ \bm{c}\ell(\tau)\big)^2d\tau\\
&\leq  \bm{c}+ \bm{c}\,\beta^{-1}\big\| \|\dot{\Psi}\|\, \ell\big\|_{2\delta}^2,
\end{split}}
\end{equation}
where we used \eqref{bound-on-xzw} to find the second inequality, and the last inequality follows from $ \dot{\Psi}\in\call_2 $ and $ \beta>1 $. Similarly,  we analyze the term $\beta^{-\frac{1}{2}} \|\Psi \dot{x}\|_{2\delta} $ on the right-hand side of \eqref{psi-x-bnd}. First, note from the pseudo-gradient dynamics \eqref{x-psi-relat}, the upper bound on $ \|x\|  $ in \eqref{bound-on-xzw}, and $ \Psi\in\call_\infty $ that $ \|\dot{x}\|\leq  \bm{c}+ \bm{c}\,\ell $. Therefore, we obtain
\begin{equation}\label{bnd-psi-dx}
{\small\begin{split}
\beta^{-1}\|\Psi \dot{x}\|_{2\delta}^2&\leq \beta^{-1}\int_{0}^{t}\!\!\exp\big(\!-\delta(t-\tau)\big)\|{\Psi}(\tau)\|^2 \|\dot x(\tau)\|^2d\tau\\
&\leq \beta^{-1}\int_{0}^{t}\!\!\exp\big(\!-\delta(t-\tau)\big)\|{\Psi}(\tau)\|^2 \big( \bm{c}+ \bm{c}\ell(\tau)\big)^2d\tau\\
&\leq  \bm{c}+ \bm{c}\,\beta^{-1}\big\|  \ell\big\|_{2\delta}^2,
\end{split}}
\end{equation}
where the last relation follows from $ \Psi\in\call_\infty $ and $ \beta>1 $.  We now consider the term $\beta^{\frac{1}{2}}\|ex^\top x\|_{2\delta}$. Following similar steps to $ \|\dot{\Psi}x\|_{2\delta} $, we use $ e\|x\| \in \call_2 $ to obtain
\begin{equation}\label{bnd-exx}
\beta \|ex^\top x\|_{2\delta}^2\leq  \bm{c}\,\beta + \bm{c}\,\beta \big\| \|e\|\|x\|\, \ell \big\|_{2\delta}^2.
\end{equation}
Lastly, we have $ \|e\|_{2\delta}\leq  \bm{c} $ since $ e\in\call_2 $.

Having found the relations \eqref{bnd-dpsi-x}, \eqref{bnd-psi-dx}, \eqref{bnd-exx}, and $ \|e\|_{2\delta}\leq  \bm{c} $, we conclude from \eqref{psi-x-bnd} and \eqref{l-bnd} that for any $ \delta\in[0,\min\{1,2\alpha_0\}) $, the following implication holds for all $ \beta>1 $:
\begin{equation*}
\ell^2\leq \bm{c}\,\beta^3+\bm{c}\,\beta^{-1} (\big\| \|\dot{\Psi}\|\, \ell\big\|_{2\delta}^2+\|\ell\|_{2\delta}^2)+\bm{c}\, \beta \big\| \|e\|\|x\|\, \ell \big\|_{2\delta}^2.
\end{equation*}
The above inequality provides an implicit upper bound of $ \ell $ which includes the $ \call_2 $ signals $ \dot{\Psi},e\|x\| $. Next, we obtain an explicit upper bound of $\ell $ to conclude its boundedness. For that, we use the definition of the exponentially weighted $\call_2 $ norm to deduce that 
\begin{equation*}
\ell^2(t)\leq \bm{c}\,\beta^3+\bm{c} \int_{0}^t \exp\big(-\delta(t-\tau)\big)k(\tau)\ell^2(\tau)d\tau,\quad \forall t\geq 0,
\end{equation*}
where
\begin{align}\label{k-def}
k(\tau):=\beta^{-1}\|\dot{\Psi}(\tau)\|^2+\beta^{-1}+ \beta \|e(\tau)\|^2\|x(\tau)\|^2.
\end{align}
It then follows from Bellman-Gronwall Lemma \cite[Lem. 3.3.9]{ioannou2012robust} that
\begin{equation}\label{l-explicit}
\ell^2(t)\leq \bm{c}\,\beta^3 \Phi(t,0)+\bm{c}\,\beta^3 \delta \int_{0}^{t}\Phi(t,\tau)d\tau,\quad \forall t\geq 0,
\end{equation}
where
\begin{equation*}
\Phi(t,\tau):=\exp\Big(-\delta(t-\tau)+\bm{c}\int_{\tau}^{t}k(s)ds\Big).
\end{equation*}
Since $ \dot{\Psi},e\|x\|\in\call_2 $, we obtain from \eqref{k-def} that
\begin{equation*}
\bm{c}\int_{\tau}^{t}k(s)ds\leq \bm{c}(\beta^{-1}+\beta)+\bm{c}\,\beta^{-1}(t-\tau),
\end{equation*}
and this implies that
\begin{equation*}
\Phi(t,\tau)\leq \bm{c}^{(\beta^{-1}+\beta)}\exp\big(-(\delta-\bm{c}\,\beta^{-1}) (t-\tau)\big).
\end{equation*}
Now select $\delta>0$ in the interval $(0,\min\{1,2\alpha_0\})$. Also
note that $ \beta>1 $ can be selected independent of $\delta$ and $ \bm{c} $ is oblivious of  $ \beta $. Thus, we choose $ \beta $  sufficiently large such that $ \delta-\bm{c}\,\beta^{-1}>0 $. It then follows from the inequality above and  \eqref{l-explicit}  that $ \ell\in\call_\infty $. Therefore, bearing \eqref{bound-on-xzw} in mind, we conclude that all  signals of the closed-loop system are uniformly bounded.
\squarend

\end{document}